\documentclass[draft]{article}
\usepackage[english]{babel}
\usepackage[utf8]{inputenc}
\usepackage{mathbbol,amsmath,amsthm,amssymb,yhmath,xspace,ifdraft}
\usepackage{easyReview}
\usepackage[a4paper,margin=3cm]{geometry}

\DeclareSymbolFontAlphabet{\mathbb}{AMSb}
\DeclareSymbolFontAlphabet{\mathbbl}{bbold}

\newcommand{\mathd}{\mathop{}\!\mathrm{d}}
\newcommand{\1}{\operatorname{\mathbbl{1}}}
\renewcommand{\i}{\operatorname{\mathbbl{i}}}
\renewcommand{\j}{\operatorname{\mathbbl{j}}}
\renewcommand{\k}{\operatorname{\mathbbl{k}}}
\newcommand{\U}{\mathcal{U}}
\newcommand{\V}{\mathcal{V}}
\renewcommand{\Re}{\mathop{\textup{Re}}}
\renewcommand{\Im}{\mathop{\textup{Im}}}

\newtheorem{theorem}{Theorem}

\newtheorem{definition}{Definition}
\newtheorem{lemma}{Lemma}

\newtheorem{remark}{Remark}

\begin{document}

\title{ Discrete CMC surfaces in $\mathbb{R}^3$ and \\
discrete minimal surfaces in $\mathbb{S}^3$.\\
A discrete Lawson correspondence}
\author{Alexander I. Bobenko, Pascal Romon}

\maketitle
\abstract{The main result of this paper is a discrete Lawson correspondence between discrete CMC surfaces in ${\mathbb R}^3$ and discrete minimal surfaces in ${\mathbb S}^3$. 
This is a correspondence between two discrete isothermic surfaces. 
We show that this correspondence is an isometry in the following sense: it preserves the metric coefficients introduced previously  by Bobenko and Suris for isothermic nets. Exactly as in the smooth case, this is a correspondence between nets with the same Lax matrices, and the immersion formulas also coincide with the smooth case.}

\section{Introduction}

\medskip

The Lawson correspondence states \cite{Lawson} that for any minimal surface in ${\mathbb S}^3$ there exists an isometric constant mean curvature surface in ${\mathbb R}^3$. It is an important tool for the investigation and construction of CMC surfaces. In particular it was a crucial tool for the classification of trinoids in \cite{KGBKS} and for the numerical construction of examples of CMC surfaces with higher topology in \cite{KGBP}. For the last purpose it was important to integrate it once and to formulate it terms of the corresponding frames \cite{OP} (see also Theorem~\ref{thm:Lawson_smooth}).

Although discrete CMC surfaces in ${\mathbb R}^3$ have been known for a longtime already \cite{BP}, as well as discrete minimal surfaces in ${\mathbb S}^3$ \cite{BHL, BuHRS}, the discrete Lawson correspondence has remained a challenge. The main problem was to define a proper discrete analogue of isometry. 

The main result of this paper is a discrete Lawson correspondence between discrete CMC surfaces in ${\mathbb R}^3$ and discrete minimal surfaces in ${\mathbb S}^3$ formulated in Theorem~\ref{thm:discrete_Lawson}. This is a correspondence between two discrete isothermic surfaces. We show that it is an isometry in the following sense: it preserves the metric coefficients for isothermic nets introduced previously in \cite{BS}. Exactly as in the smooth case, this is a correspondence between nets with the same Lax matrices, and the immersion formulas also coincide with the smooth case. As necessary intermediary results, we show that commuting Lax pairs generate discrete CMC and minimal surfaces via an immersion formula as in the smooth case (Theorems \ref{thm:CMCR3} and \ref{thm:CMCS3}), and, conversely, that all discrete CMC and minimal surfaces are generated by Lax pairs (Theorem~\ref{thm:reconstruction}).

Another approach to discrete isothermic surfaces in spaceforms via conserved quantities 
and discrete line bundles has been proposed in \cite{BuHRS,BuHR}, encompassing constant mean 
curvature nets as a special case. A Calapso transformation is defined therein, 
which generalizes the Lawson correspondence between the relevant spaceforms and preserves 
geometric quantities. In contrast, our more pedestrian method focuses on the immersion formulas 
which are proven to be identical as in the smooth case, and provide a more explicit 
definition of the correspondence and its metric invariance. Still, both definitions agree,
as we prove in Remark~\ref{Calapso}.
\medskip

We shall consider in this paper only meshes with quadrilateral planar faces,
known as quad-nets or Q-nets for short (also called PQ-meshes), whose
theoretical properties mimic those of their smooth counterparts. In the
particular case of nets indexed by $\mathbb{Z}^2$, indices play the same role
as coordinates of an immersion, and specific choices of Q-nets correspond to
specific parametrizations of surfaces.

Throughout the text, we will use the shift notation to describe the local
geometry: When $F$ is a net, $F = F_0 = F (0, 0)$ will denote a base point,
while $F_1, F_2, F_{12}$ will stand for $F (1, 0), F (0, 1), F (1, 1)$, so
that indices $1, 2$ correspond to shift in the first and second variables
respectively. The same holds for any vertex-based function. The edges of the
face $(F, F_1, F_{12}, F_2)$ are labeled $(0, 1)$, $(1, 12)$, $(12, 2)$ and
$(2, 0)$ and the values of an edge-based function $u$ will be denoted by
$u_{01}$, $u_{1, 12}$, etc. If $(i, j)$ is a pair of indices corresponding to
an edge, $\mathd \varphi_{i j}$ is by definition $\varphi_j - \varphi_i$.

\medskip

\emph{Acknowledgments}: the Authors wish to thank Udo Hertrich-Jeromin and 
Wayne Rossman for the fruitful discussions on this topic, and Tim Hoffmann for
his judicious remarks.


\section{The smooth theory}	\label{sec:smooth-theory}

\subsection{Constant mean curvature surfaces in $\mathbb{R}^3$}

We recall here a well known (see for example \cite{B, BP} for more details) description of CMC surfaces in $\mathbb{R}^3$ in terms of loop groups and quaternionic frames. The normalizations used in the present paper coincide with the normalizations in \cite{BP}.

In the sequel we identify the Euclidean three space $\mathbb{R}^3$ with imaginary quaternions ${\rm Im}\ {\mathbb{H}}$, and the standard imaginary quaternions with an orthonormal basis of $\mathbb{R}^3$,
and use the following matrix representation: 
\begin{eqnarray} 
\i=\left( \begin {array}{cc} 0&-i\\-i&0 \end{array}\right),
\j=\left( \begin {array}{cc} 0&{-1}\\1&0 \end{array}\right),
\k=\left( \begin {array}{cc} -i&0\\0&i \end{array}\right),
\1 =\left( \begin {array}{cc} 1&0\\0&1 \end{array}\right).     \label{AC1.3}
 \end{eqnarray}
This results in the following matrix representation of vectors in $\mathbb{R}^3$:
\begin{eqnarray}
X = (X_1,X_2,X_3) \ \longleftrightarrow \ \left(
\begin{array}{cc}
-i X_3 & -i X_1 -X_2 \\ 
-i X_1 +X_2 & i X_3 \end{array} \right) \cdotp                        \label{AC1.5}
\end{eqnarray}

Umbilic free CMC-surfaces are isothermic (conformal curvature line parametrization). Let $(x,y)\mapsto F(x,y)$ be a CMC isothermically parametrized surface. Without loss of generality, one can normalize the mean curvature $H=1$ and the Hopf differential $Q=\langle F_{xx}-F_{yy}+2iF_{xy},N\rangle=\frac{1}{2}$. Let $e^u$ be the corresponding conformal metric: $\langle dF,dF \rangle = e^u (dx^2+dy^2)$. It satisfies the elliptic sinh-Gordon equation
\begin{equation} \label{eq:sinh-Gordon}
u_{xx}+u_{yy}+\sinh u=0.
\end{equation}
The quaternionic frame $\Phi$ is defined as a solution of the system
\begin{eqnarray}
&\Phi_x = U\Phi,\qquad 
\Phi_y= V\Phi,              \label{H1.6}
\end{eqnarray}
where
\begin{eqnarray}
&U=\displaystyle {1\over 2}\left( \begin {array}{cc}
-\displaystyle{\frac {i}{2}}u_y & -\lambda e^{-u/2}-\displaystyle{1\over \lambda}e^{u/2}\\
\lambda e^{u/2}+\displaystyle{1\over \lambda}e^{-u/2}  &  \displaystyle{\frac {i}{2}}u_y  
\end{array} \right),                \nonumber\\
&                                    \label{H1.7}\\  
&V=\displaystyle {1\over 2}\left( \begin {array}{cc}
\displaystyle{\frac {i}{2}}u_x & -i\lambda e^{-u/2}+\displaystyle{i\over \lambda}e^{u/2}\\
i\lambda e^{u/2}-\displaystyle{i\over \lambda}e^{-u/2}  &  -\displaystyle{\frac {i}{2}}u_x    
\end{array}\right).                 \nonumber         
\end{eqnarray}
The matrices \eqref{H1.7} belong to the loop algebra
$$
g_H[\lambda ] =\{ \xi:S^1\to su(2): 
\xi (-\lambda)=\sigma_3 \xi (\lambda)\sigma_3 \} \, ,
\text{ where }
\sigma_3 = i \k = \begin{pmatrix} 1 & 0 \\ 0 & -1 \end{pmatrix}, 
$$
and $\Phi$ in \eqref{H1.6} lies in the corresponding loop group
\begin{equation}
G_H[\lambda ] =\{ \phi:S^1\to SU(2): 
\phi (-\lambda)=\sigma_3 \phi (\lambda)\sigma_3 \}.     \label{H1.8}
\end{equation}
Here $S^1$ is the set $| \lambda |=1$.

The system \eqref{H1.7} is the Lax representation for \eqref{eq:sinh-Gordon}, where the parameter $\lambda$ is called the spectral parameter.

\begin{theorem} \label{th:smooth_R^3}
The formulas
\begin{equation}   \label{eq:smooth_immersion_R3}  
	\quad \hat{N} =-\check{N}= - \Phi^{- 1} \k \Phi, \quad 
	\left\{ \begin{array}{ll}
	\hat{F} & = - \Phi^{- 1}  \frac{\partial \Phi}{\partial \gamma} - \frac{1}{2}  \hat{N}
   \\
   \check{F} & = - \Phi^{- 1} \frac{\partial \Phi}{\partial \gamma} +
   \frac{1}{2}  \hat{N} = \hat{F} + \hat{N} 
   \end{array} \right.,
\end{equation}
where $\lambda=e^{i\gamma}$, describe two parallel surfaces $\hat{F}, \check{F}$ with constant 
mean curvature $H=1$ and their Gauss maps $\hat{N}, \check{N}$.
Variation of $\gamma$ is an isometry, and the corresponding one parameter family of CMC surfaces is called the {\em associated family}. For $\gamma=0$, i.e. $\lambda=1$, the parametrizations of $\hat{F}$ and  $\check{F}$ are isothermic.
\end{theorem}

\subsection{Constant mean curvature and minimal surfaces in $\mathbb{S}^3$}

The same Lax pair yields a CMC net in $\mathbb{S}^3$ through the immersion formula obtained in \cite{B}.
We identify $S^3$ with unitary quaternions
\begin{eqnarray}
X = (X_1,X_2,X_3,X_4) \ \longleftrightarrow \ \left(
\begin{array}{cc}
-i X_3 +X_4 & -i X_1 -X_2 \\ 
-i X_1 +X_2 & i X_3+X_4 \end{array} \right) \cdotp                        \label{eq:quaternionic_S^3}
\end{eqnarray} 
If we gauge the frame into 
$\Psi = \begin{pmatrix}  e^{i \gamma / 2} & 0 \\ 0 & e^{- i \gamma / 2} 
\end{pmatrix} \Phi = \exp \left( - \frac{\gamma}{2} \k \right) \Phi$, 
then for any pair $\lambda_1 = e^{i \gamma_1}, \lambda_2 = e^{i
\gamma_2}$ in the unit circle,
\begin{equation*}  
F = \Psi (\lambda_1)^{- 1} \Psi (\lambda_2), \quad 
N = - \Psi (\lambda_1)^{- 1} \k \Psi (\lambda_2) 
\end{equation*}
are an orthogonal pair of vectors in $\mathbb{S}^3$. They describe a surface $F$ with constant mean curvature $H=\cot(\gamma_1-\gamma_2)$ and its Gauss map $N$.
In terms of the original frame the formulas look as follows:  
\begin{equation}   \label{eq:smooth_S^3}
	F = \Phi (\lambda_1)^{- 1} M \Phi (\lambda_2), \; 
	N = - \Phi (\lambda_1)^{- 1} \k M \Phi (\lambda_2),
\end{equation}
where $M = \exp (\frac{\gamma_1 - \gamma_2}{2} \k)$.

\begin{theorem} \label{th:smooth_S^3}
Let $\Phi(\lambda)$ be a solution of \eqref{H1.6}. Formulas \eqref{eq:smooth_S^3} describe a surface $F$ with constant mean curvature $H=\cot(\gamma_1-\gamma_2)$ and its Gauss map $N$. The parametrization $F(x,y)$ is isothermic if and only if $\lambda_2 = \pm \lambda_1^{-1}$ ($\gamma_1 + \gamma_2 \equiv 0 \mod \pi$). 
In particular, for $\gamma_1=-\gamma_2=\frac{\pi}{4}$ one obtains an isothermically parametrized minimal surface $F$ with the Gauss map $N$ and the conformal metric $e^{-u}$. Equivalently they can be treated as an isothermically parametrized minimal surface $N$ with the Gauss map $F$ and the conformal metric $e^{u}$.
\end{theorem}

\subsection{The Lawson correspondence}

As we have indicated already in Theorem~\ref{th:smooth_S^3}, CMC surfaces in $\mathbb{R}^3$ and minimal surfaces in $\mathbb{S}^3$ corresponding to the same Lax pair are isometric. This correspondence can be lifted to the frames without referring to the Lax representation. The corresponding formulas were obtained in \cite{OP}. We will derive them from the immersion formulas \eqref{th:smooth_R^3}, \eqref{eq:smooth_S^3}. 

\begin{theorem}\label{thm:Lawson_smooth}
Let $F$ and $N$ be a pair of Christoffel dual ($F^*=N$) isothermic surfaces in $\mathbb{S}^3$. They can be treated equivalently as the minimal surface $F$ with the Gauss map $N$ or the minimal surface $N$ with the Gauss map $F$. 
Then there exist surfaces $\hat{F}$ and $\check{F}$ in $\mathbb{R}^3$ with constant mean curvature $H=1$ isometric to $F$ and $N$ respectively. They and their Gauss maps $\hat{N}$ and $\check{N}$  are given by the following formulas:   
\begin{eqnarray}
d\hat{F}= F^{-1}*dF, \quad \hat{N}=F^{-1}N=-NF^{-1} \nonumber\\
d\check{F}=N^{-1}*dN, \quad \check{N}=-F^{-1}N=NF^{-1} \label{eq:Lawson_smooth}.
\end{eqnarray}
Here $*$ is the Hodge star defined by ($*f_x=-f_y, *f_y=f_x$).
The parametrization of surfaces $\hat{F}, \check{F}$ given by \eqref{eq:Lawson_smooth} inherited from the isothermic parametrization of $F$ and $N$ is not isothermic. Formulas \eqref{eq:Lawson_smooth} give the surfaces from the associated family \eqref{eq:smooth_immersion_R3} corresponding to $\lambda=\lambda_2= e^{-i\frac{\pi}{4}}$.

Surfaces $F,N$ in $\mathbb{S}^3$ and surfaces $\hat{F},\check{F}$ in $\mathbb{R}^3$ with the Gauss maps $\hat{N},\check{N}$ are described by the formulas \eqref{eq:smooth_S^3} and \eqref{eq:smooth_immersion_R3} with the same Lax matrices \eqref{H1.7}.
\end{theorem}

\begin{proof}
Follows from direct computation. In the minimal surface case \eqref{eq:smooth_S^3} becomes
\begin{eqnarray} 
F = \Phi (\lambda_1)^{- 1} M \Phi (\lambda_2), \; 
	N =  \Phi (\lambda_1)^{- 1} M^{-1}\Phi (\lambda_2),\;
	M=\left( \begin {array}{cc}
e^{-i\frac{\pi}{4}} & 0\\
0  &  e^{i\frac{\pi}{4}} 
\end{array} \right).
\end{eqnarray}

Formulas for the Gauss maps follow immediately:
$$
F^{-1}N=\Phi^{-1}(\lambda_2)M^{-2}\Phi(\lambda_2)=\hat{N}(\lambda_2). 
$$
Computations for $d\hat{F}$ and $d\check{F}$ are slightly more involved,
$$
d\hat{F}=-d(\Phi)^{-1}\frac{\partial\Phi}{\partial\gamma}-\Phi^{-1}\frac{\partial d\Phi}{\partial\gamma}=
\Phi^{-1}\left(
-\frac{\partial}{\partial\gamma}(Udx+Vdy)+\frac{1}{2}[ \k,Udx+Vdy]
\right)\Phi.
$$
Calculating at $\lambda_2=e^{-i\frac{\pi}{4}}$ we get
\begin{equation} \label{eq:dF^}
d\hat{F}=e^{-u/2}\Phi^{-1}(\lambda_2)\left(
\left( \begin {array}{cc}
0 & e^{i\frac{\pi}{4}}\\
-e^{-i\frac{\pi}{4}}  & 0 
\end{array} \right)dx+
\left( \begin {array}{cc}
0 & -e^{-i\frac{\pi}{4}}\\
e^{i\frac{\pi}{4}}  & 0 
\end{array} \right)dy
\right)\Phi(\lambda_2).
\end{equation}
On the other hand from the formulas for surfaces in $\mathbb{S}^3$ we obtain
$$
dF=\Phi^{-1}(\lambda_1)\left(
(U(\lambda_1)dx-V(\lambda_1)dy)M+M(U(\lambda_2)dx+M(\lambda_2)dy)
\right)\Phi(\lambda_2),
$$
which implies
\begin{eqnarray*}
&*dF=\Phi^{-1}(\lambda_1)\left(
(V(\lambda_1)M-MV(\lambda_2))dx+(-U(\lambda_1)M+MU(\lambda_2))dy
\right)\Phi(\lambda_2), \\
&F^{-1}*F=\Phi^{-1}(\lambda_2)\left(
(M^{-1}V(\lambda_1)M-V(\lambda_2))dx+(-M^{-1}U(\lambda_1)M+U(\lambda_2))dy
\right)\Phi(\lambda_2) 
\end{eqnarray*}
A direct computation shows that the last expression coincides with \eqref{eq:dF^}.
The identity for $d\check{F}$ follows in the same way.
\end{proof}


\section{Discrete CMC and minimal surfaces in $\mathbb{R}^3$ and~$\mathbb{S}^3$}
\label{sec:discrete_CMC_Q-nets}

We will now define the discrete analogs of constant mean curvature and minimal surfaces, 
following the Steiner formula approach of \cite{BPW} and \cite{BHL}.

\medskip
Let $(F, N)$ be a pair of edge-parallel
maps from $\mathbb{Z}^2$ to $\mathbb{R}^4$, with planar faces, where either
\begin{itemize}
  \item  $F$ lies in $\mathbb{R}^3$ and $N$ takes values in $\mathbb{S}^2$, or,
  
  \item $F$ and $N$ lie in $\mathbb{S}^3$ and $F \bot N$ at each vertex.
\end{itemize}
The map $N$ is treated as the Gauss map of $F$. Since $N$ is planar and
constrained to a sphere, its faces are circular, and so are those of $F$, by
parallelism. Note that circular implies planar.

The area $\mathcal{A} (f)$ of a planar face $f$ being a quadratic form in its coordinates, we
define the mixed area $\mathcal{A} (f, f')$ of two edge-parallel faces to be the polar
form applied to $f, f'$:
\[ \mathcal{A} (f, f') = \frac{1}{4}  (\mathcal{A} (f + f') -\mathcal{A} (f -
   f')) \]
where $\mathcal{A} (f) = \mathcal{A} (f, f)$. This allows us to write a Steiner formula for the
area of the parallel face $f + \varepsilon f'$ as
\[ 
\mathcal{A} (f + \varepsilon f') =\mathcal{A} (f) + 2 \varepsilon
\mathcal{A} (f, f') + \varepsilon^2 \mathcal{A} (f')\, . 
\]
Applying this formula to the mesh pair $(F, N)$ on the face $f$, we identify
the mean and Gaussian curvature by $\mathcal{A} (F (f) + \varepsilon N (f)) 
= (1 - 2 \varepsilon H_f + \varepsilon^2 K_f) \mathcal{A} (F (f))$,
so that
\[ H_f = - \frac{\mathcal{A} (F (f), N (f))}{\mathcal{A} (F (f))} \text{ and }
   K_f = \frac{\mathcal{A} (N (f))}{\mathcal{A} (F (f))} \;\cdot 
\]
\begin{definition}
  A circular Q-net $(F, N)$, with $F,N$ as above, is of \emph{constant mean curvature}
  $H \neq 0$ (CMC) if $H_f = H$ on all faces $f$. It is \emph{minimal} if $H_f$ vanishes
  identically.
\end{definition}

Such a net is automatically {\emph{Koenigs}} (see \cite{BS}), i.e. it possesses a Christoffel
dual $F^{\ast}$ such that (i) $F^{\ast}$ is edge-parallel to $F$ and (ii) $A
(F, F^{\ast})$ vanishes identically. Indeed, \ if $(F, N)$ has constant mean
curvature $H$ (resp. is minimal), then $F^{\ast} = F + \frac{1}{H} N$ (resp.
$F^{\ast} = N$) is the dual. 
Being Koenigs and circular is equivalent for $F$ to be {\emph{discrete
isothermic}}. Discrete isothermic nets were originally defined in \cite{BP_isothermic} as nets with factorisable cross ratios, i.e. the cross ratio $\text{cr} (F, F_1, F_{12}, F_2)$ is of the
form $A / B$, with $A$ depending 
on the first coordinate and $B$ on
the second. Such functions $A, B$ are called {\emph{edge labelings}} and are uniquely defined up to a common factor.

In \cite{BS, BS-Koenigs} a discrete analogue of conformal metric was introduced for discrete isothermic surfaces.
It was shown that Koenigs nets possess a function $s:{\mathbb Z}^2\to{\mathbb R}^+$ defined at vertices, 
called the {\emph{(discrete conformal) metric coefficient}}. Consider black and white sublattices of ${\mathbb Z}^2$ so that every elementary quad contains two vertices of each displaced diagonally.

The conformal factor $s$ is defined up to a so called black-white rescaling: $s\mapsto \lambda s$ at black points, and $s\mapsto \mu s$ at white points. 
In particular $s$ relates the net to its Christoffel dual: 
\begin{equation}\label{eq:dual_through_s}
 F^{\ast}_i - F^{\ast} = \frac{1}{s_i s} (F_i - F), \quad i=1,2.
\end{equation}
Moreover, for discrete isothermic nets the edge labeling\footnote{Edge labelings are unique up to global
multiplication of $A$ and $B$ by a constant. The choice mentioned here is canonical.}
is linked to the discrete conformal factor $s$ and the edge lengths as follows (see \cite{BS, BS-Koenigs}): 
\begin{equation}\label{eq:AB_labeling}
A= \frac{\| F_1-F\|^2}{s s_1}, \ B= \frac{\| F_2-F\|^2}{s s_2}.
\end{equation}

One can approximate smooth isothermic surfaces by discrete isothermic surfaces \cite{BuM}. Probably this is also the case with minimal and CMC surfaces, although this is not yet proven.
%


\section{Loop group description}

Here following \cite{BP} we present the loop group description of discrete CMC surfaces in ${\mathbb R}^3$. We will show also that discrete CMC surfaces in ${\mathbb S}^3$ are described by the same discrete Lax representation and the immersion formula (\ref{eq:smooth_S^3}) of the smooth case.

\subsection{Discretization in the loop group}

As in the smooth case, we consider a frame $\Phi : \mathbb{Z}^2 \to G_H[\lambda ]$.
The discrete Lax pair $\U (\lambda) = \Phi_1(\lambda) \Phi (\lambda)^{- 1}$, 
$\V (\lambda) = \Phi_2(\lambda) \Phi (\lambda)^{- 1}$
are maps from the edges into the loop group. 
By analogy with the smooth immersions $\U (\lambda)$, $\V (\lambda)$ are defined of the following form:
on each edge,
\begin{equation}   \label{eq:Lax-pair}
\begin{split}
  \U (\lambda) & = \frac{1}{\alpha (\lambda)}  \begin{pmatrix}
    a & - \lambda u - \lambda^{- 1} u^{- 1}\\
    \lambda u^{- 1} + \lambda^{- 1} u & \bar{a}
  \end{pmatrix},
\\
	\V (\lambda) &= \frac{1}{\beta
  (\lambda)}  \begin{pmatrix}
    b & - i \lambda v + i \lambda^{- 1} v^{- 1}\\
    i \lambda v^{- 1} - i \lambda^{- 1} v & \bar{b}
  \end{pmatrix},  \end{split}
\end{equation}
where complex valued $a, b$ and real valued $u, v$ do not depend on $\lambda$, $u, v$ are positive and
$\alpha (\lambda)$ and $\beta (\lambda)$ are real such that the determinants are equal to $1$:
\begin{equation}\label{eq:alpha2beta2}
  \alpha (\lambda)^2 = | a |^2 + \lambda^2 + \lambda^{- 2} + u^2 + u^{- 2},
  \quad \beta (\lambda)^2 = | b |^2 - \lambda^2 - \lambda^{- 2} + v^2 +
  v^{- 2}. 
\end{equation}
Furthermore, $\alpha$ and $\beta$ on the opposite edges coincide, i.e. they are edge labeling for the
first and second indices respectively.

\medskip	

We will now focus on a single quad $(F, F_1, F_{12}, F_2)$, and let
$\U, \V$ be the Lax matrices associated to the edges $(F, F_1)$
and $(F, F_2)$ respectively; for the sake of simplicity, we will mark with a prime
the corresponding quantities on the opposite edges $(F_2, F_{12})$ and $(F_1,
F_{12})$: $\U'$, $\V'$, $a'$, $u'$, etc. In particular $\Phi_{12}=\U' \Phi_2$, $\Phi_{12}=\V' \Phi_1$. Note that $\alpha' = \alpha$ and $\beta' = \beta$.

The Lax pair satisfies 
\begin{equation} \label{eq:discrete_Lax}
\V'(\lambda) \, \U(\lambda) = \U'(\lambda) \, \V(\lambda)
\end{equation}
on any quad and gives rise to a frame $\Phi (\lambda) : \mathbb{Z}^2
\rightarrow G_H[\lambda ]$.

This commutation property yields
\begin{equation}
  uu' = vv' \label{eq:uuprimevvprime}
\end{equation}
\begin{equation}
  b' a - ba' = i (u' v + uv' - u'^{- 1} v^{- 1} - u^{- 1} v'^{- 1})
\end{equation}
\begin{equation} \label{eq:commut1}
  \bar{b}u' - b' u = i (\bar{a}v' - a' v )
\end{equation}
\begin{equation} \label{eq:commut2}
  \bar{b}u'^{- 1} - b' u^{- 1} = i (a' v^{- 1} - \bar{a} v'^{- 1})
\end{equation}
As noticed in \cite[(4.23)]{BP}, equation \eqref{eq:uuprimevvprime} is equivalent
to the existence of a vertex function $w$ such that
\begin{equation}
  u = w w_1, \quad u' = w_2 w_{12}, \quad v = w w_2, \quad
  v' = w_1 w_{12} \label{eq:w}
\end{equation}

The function $w$ turns out to be essentially the discrete conformal metric $s$, as we will show in the next section.


\subsection{Discrete CMC nets in Euclidean three space}
\label{sec:discrete_CMC_R3}

Let $\Phi(\lambda)$ be a frame defined from commuting Lax pairs as above, and let 
$\lambda = e^{i \gamma} \in \mathbb{S}^1$ be a spectral parameter. 
We define two nets $\hat{F}, \check{F}$ and a unit Gauss map~$\hat{N}$ as follows:
\begin{equation}   \label{eq:discrete_immersion_R3}  
	\quad \hat{N} = - \Phi^{- 1} \k \Phi, \quad 
	\left\{ \begin{array}{ll}
	\hat{F} & = - \Phi^{- 1}  \frac{\partial \Phi}{\partial \gamma}_{| \gamma = 0
   } - \frac{1}{2}  \hat{N}
   \\
   \check{F} & = - \Phi^{- 1} \frac{\partial \Phi}{\partial \gamma}_{| \gamma = 0 } +
   \frac{1}{2}  \hat{N} = \hat{F} + \hat{N} 
   \end{array} \right.
\end{equation}
where all the matrices are evaluated at $\gamma = 0$ (i.e. $\lambda = 1$).

\begin{theorem}[\cite{BP}] \label{thm:CMCR3} \ \\
  The pair $(\hat{F}, \hat{N})$ given by~\eqref{eq:discrete_immersion_R3} 
  is a CMC net in $\mathbb{R}^3$ with $H=1$. Its Christoffel dual is
  $\check{F}$. 
  On any quad the discrete conformal metric $s$ is given by
  \begin{equation}
    ss_1 = - u^2 \text{ and } ss_2 = v^2, \label{eq:vertexmetricR3}
  \end{equation}
  and the cross ratio  $\text{cr} (\hat{F}, \hat{F}_1, \hat{F}_{12}, \hat{F}_2)$ is equal to $- \beta(1)^2 / \alpha(1)^2$.
  The edge lengths are equal
  \begin{align*} 
  \| \hat{F}_1 - \hat{F} \|^2 &= \frac{4 u^2}{\alpha(1)^2} 
  = - \frac{4 s s_1 }{| a |^2 + 2 - ss_1 - s^{- 1} s_1^{- 1}} \, ,
  \\
  \| \hat{F}_2 - \hat{F} \|^2 &= \frac{4 v^2}{\beta(1)^2} 
  = \frac{4 ss_2 }{| b |^2 - 2 + s s_2 + s^{- 1} s_2^{- 1}}
  \, \cdot 
  \end{align*}
\end{theorem}

\begin{proof}
	Note that \cite{BP} use a slightly different, albeit equivalent notation. 
	For the sake of completeness and compatibility with the spherical case, 
	we shall sketch the proof here using our notations. 
	Note: all $\lambda$-dependent quantities are evaluated at $\lambda = 1$, 
	and we shall not write the variable $\lambda$ for greater legibility,
	so $\alpha$ stands for $\alpha(1)$, etc.
	
  Evaluating the edge vectors one obtains
  \begin{align}  \label{eq:def-U}
    \hat{F}_1 - \hat{F}
    & = - \Phi^{- 1} \U^{- 1}  \left( \dot{\U} -
    \frac{1}{2}  [\k, \U] \right) \Phi,
    &
    \hat{F}_2 - \hat{F} & = - \Phi^{- 1} \V^{- 1}  \left(
    \dot{\V} - \frac{1}{2}  [\k, \V] \right)
    \Phi,
    \\
    \hat{N}_1 - \hat{N} & =  - \Phi^{- 1} \U^{- 1}  [\k, \U] \Phi,
    &
    \hat{N}_2 - \hat{N} & = - \Phi^{- 1} \V^{- 1}  [\k, \V] \Phi,
  \end{align}
  where $\dot{\U}$ the derivative at $\gamma = 0$, and similar formulas for $\check{F}$.
Because $\alpha$ and $\beta$ have an extremum at $\lambda = 1$, $\dot{\U} = \frac{u - u^{- 1}}{\alpha} \i$ and $\dot{\V} = - \frac{v + v^{- 1}}{\beta} \j$. Since $[\k, \U] = - \frac{2 (u + u^{- 1})}{\alpha} \i$ and $[\k, \V] = \frac{2 (v - v^{- 1})}{\beta} \j$,
%
parallelism between the nets $\hat{F}$, $\hat{N}$ (and thus $\check{F}$) is clear,
as are the lengths of the edges; we also derive the 
proportionality factors
  \[ \mathd \check{F}_{01} = - u^{- 2} \mathd \hat{F}_{01}, \quad
     \mathd \check{F}_{02} = v^{- 2} \mathd \hat{F}_{02}, \]
  and using \eqref{eq:w}
  \[ \mathd \check{F}_{01} = - \frac{\mathd \hat{F}_{01}}{w^2 w_1^2},
     \quad \mathd \check{F}_{02} = \frac{\mathd \hat{F}_{02}}{w^2
     w_2^2} \cdot \]
  That proves that $\hat{F}$ is Koenigs with Christoffel dual $\check{F}$ and
  a discrete conformal metric given by $s = \pm w^2$ (see \cite{BS} for details). The sign can be chosen constant in one direction and alternates in other direction.

The claim about the cross ratios is proven by direct computation using the quaternionic formulas for cross ratios (see \cite{BP}).
\end{proof}

\medskip

As noted in the proof, the choice of $\lambda=1$ for the spectral parameter 
in~\eqref{eq:discrete_immersion_R3} is 
crucial (though not unique, any extremum of $\lambda^2 + \lambda^{-2}$ will do, i.e.
$\lambda \in \{ \pm 1, \pm i \}$). Other values of the spectral parameter, 
as in the smooth case (see Theorem~\ref{th:smooth_R^3}), will not satisfy the parallelism condition 
between the edges, nor the planarity of the faces, so that these nets are 
neither circular nor Koenigs.
Such nets 
retain interesting properties, in particular they are 
\emph{edge-constraint nets}, as defined in~\cite{HSFW}. Edge-constraint nets
are quadrilateral nets $F$ with non necessarily planar faces, and vertex normals $N$ such that, on any edge $e = (F F_i)$,
the average of the normals at the endpoints is orthogonal to the edge: 
$F_i - F \perp N+N_i$. 

Moreover, in this case the mean curvature can also be defined via a discrete version of Steiner's formula. For that purpose the area functional for non-planar quads is defined by projecting along the normal direction which is orthogonal to the diagonals of the quad. In was shown in \cite{HSFW} that all surfaces of the {\emph associated family} (i.e. those defined by (\ref{eq:discrete_immersion_R3}) with general unitary $\lambda$) possess constant mean curvature.


\subsection{Discrete CMC and minimal nets in the three sphere}

As in the case of smooth CMC surfaces, the same Lax pair leads yields a discrete CMC
surface in $\mathbb{S}^3$ through the immersion formula (\ref{eq:smooth_S^3}). 
We gauge again the frame into 
$\Psi = \begin{pmatrix}  e^{i \gamma / 2} & 0 \\ 0 & e^{- i \gamma / 2} 
\end{pmatrix} \Phi = \exp \left( - \frac{\gamma}{2} \k \right) \Phi$,
and define, for any pair $\lambda_1 = e^{i \gamma_1}, \lambda_2 = e^{i
\gamma_2}$ in the unit circle,
\begin{equation*}  
F = \Psi (\lambda_1)^{- 1} \Psi (\lambda_2), \quad 
N = - \Psi (\lambda_1)^{- 1} \k \Psi (\lambda_2) 
\end{equation*}
which are an orthogonal pair of vectors in $\mathbb{S}^3$. 
Equivalently, 
\begin{equation}  \label{eq:discrete_immersion_S3} 
	F = \Phi (\lambda_1)^{- 1} M \Phi (\lambda_2), \; 
	N = - \Phi (\lambda_1)^{- 1} \k M \Phi (\lambda_2),
\end{equation}
where $M = \begin{pmatrix} e^{i \frac{\gamma_2 - \gamma_1}{2}} & 0\\
0 & e^{- i \frac{\gamma_2 - \gamma_1}{2}} \end{pmatrix}
= \exp (\frac{\gamma_1 - \gamma_2}{2} \k)$.
\begin{theorem} \label{thm:CMCS3}
   The pair $(F, N)$ given by~\eqref{eq:discrete_immersion_S3}
  is a discrete isothermic CMC surface in $\mathbb{S}^3$ if, and only if, $\lambda_2 = \pm \lambda_1^{-1}$ ($\gamma_1 + \gamma_2 \equiv 0 \mod \pi$). Its mean curvature is equal to
  \[ H = \frac{\Re \lambda_1^2}{\Im \lambda_1^2 } = \cot (2
     \gamma_1) = \cot (\gamma_1 - \gamma_2) \, .
  \]
  Its Christoffel dual is $F^{\ast} = F + \frac{1}{H} N$ if $H\neq 0$, and 
  $F^{\ast} = N$ if $H = 0$. 
  On any quad the discrete conformal metric $s$ of $F$ is given by
  \begin{equation}
    \begin{array}{ll}
      ss_1 = - u^2  \sqrt{\frac{H^2}{1 + H^2}} \text{ and } ss_2 = v^2
      \sqrt{\frac{H^2}{1 + H^2}} & \text{ if } H \neq 0,\\
      ss_1 = - u^2 \text{ and } ss_2 = v^2 & \text{ if } H = 0,
    \end{array} \label{eq:vertex-metric-S3}
  \end{equation}
  and the cross ratio is equal to $- \beta^2 / \alpha^2$
  evaluated at $\lambda_1$. 
  The edge lengths satisfy, when $H \neq 0$,
  
  \begin{eqnarray} \label{eq:edge_lengths_H}
    \| \mathd F_{01} \|^2  =\frac{4 u^2 \sin^2  (2 \gamma_1)}{\alpha^2} =
    \frac{4 | ss_1 |}{\alpha^2 H \sqrt{1 + H^2}} 
    = \frac{4 | ss_1 |}{\left( |
    a |^2 + u^2 + u^{- 2} + \sqrt{\frac{H^2}{1 + H^2}} \right) H \sqrt{1 +
    H^2} },\nonumber\\
    \| \mathd F_{02} \|^2  = \frac{4 v^2 \sin^2  (2 \gamma_1)}{\beta^2} =
    \frac{4 | ss_2 |}{\beta^2 H \sqrt{1 + H^2}} 
    = \frac{4 | ss_2 |}{\left( | b
    |^2 + v^2 + v^{- 2} - \sqrt{\frac{H^2}{1 + H^2}} \right) H \sqrt{1 + H^2}},
  \end{eqnarray}
  and for minimal nets,
  \begin{equation} \label{eq:edge_lengths_0}
  \| \mathd F_{01} \|^2 = \frac{4 u^2}{\alpha^2} = - \frac{4 ss_1}{| a
     |^2 - ss_1 - s^{- 1} s_1^{- 1}}, \; \| \mathd F_{02} \|^2 =
     \frac{4 v^2}{\beta^2} = \frac{4 s s_2}{| a |^2 + ss_2 + s^{- 1} s_2^{-
     1}} \cdot 
    \end{equation}
\end{theorem}

\begin{proof} is given by direct computation.
  Let us check that the edges of $F$ and $N$ are parallel.
  \begin{align*}
    F_1 - F 
    & = \Phi (\lambda_1)^{- 1} \U (\lambda_1)^{- 1}  (M\U
    (\lambda_2) -\U (\lambda_1) M) \Phi (\lambda_2),
    \\
    N_1 - N 
    & = - \Phi (\lambda_1)^{- 1} \U (\lambda_1)^{- 1} 
    (\k M\U (\lambda_2) -\U (\lambda_1) M\k) \Phi (\lambda_2).
  \end{align*}
  Note that, given two unitary matrices $U_1, U_2$ in $SU (2)$, $U_2 -
  U_1$ is a real multiple of $\k U_2 - U_1 \k$ iff their
  diagonals coincide. Applying it to $U_1 =\U (\lambda_1) M$, $U_2 =
  M\U (\lambda_2)$, we conclude in particular that $\alpha
  (\lambda_1) = \alpha (\lambda_2)$; so by \eqref{eq:alpha2beta2}, $\Re
  \lambda_2^2 = \Re \lambda_1^2$. 
The case $\lambda_2 = \pm  \lambda_1$ leads to constant maps. The case $\lambda_2 = \pm \lambda_1^{- 1}$ leads to non-trivial discrete surfaces. 
  Let us focus on $\lambda_2 = \lambda_1^{- 1}$, i.e. $\gamma_2 + \gamma_1 = 0$.
  We have then 
  \[ M\U (\lambda_2) -\U (\lambda_1) M = - \frac{2 u \sin (2
     \gamma_1)}{\alpha} \i 
     \text{ and }
      -\k M\U (\lambda_2) +\U (\lambda_1)
     M\k= \frac{2}{\alpha} (u^{- 1} + u \cos (2 \gamma_1))
     \i, 
     \]
  and similarly
  \[ M\V (\lambda_2) -\V (\lambda_1) M = \frac{2 v}{\beta}
     \sin (2 \gamma_1) \j 
  \text{ and }
  -\k M\V (\lambda_2) +\V (\lambda_1) M\k
  = \frac{2}{\beta}  (v^{- 1} - v \cos (2 \gamma_1)) \j 
  \]
  which proves the parallelism of the corresponding edges. If we set $F^{\ast} = F +
  \frac{1}{H} N$, with $H = \cot (2 \gamma_1) \neq 0$, then 
  \[ 
  \mathd F^{\ast}_{01} = \mathd F_{01} + \tan (2 \gamma_1) \mathd N_{01} =
     - \frac{u^{- 2}}{\cos (2 \gamma_1)} \mathd F_{01} = - u^{- 2} 
     \sqrt{\frac{1 + H^2}{H^2}} \mathd F_{01}, \]
  \[ \mathd F^{\ast}_{02} = \mathd F_{02} + \tan (2 \gamma_1) \mathd N_{02} =
     \frac{v^{- 2}}{\cos (2 \gamma_1)} \mathd F_{02} = v^{- 2}  \sqrt{\frac{1
     + H^2}{H^2}} \mathd F_{02}. \]
  We infer the existence of a discrete conformal metric $s$, as in the proof of Theorem
  \ref{thm:CMCR3}. Note that
  \[ 
  s = \pm w^2  \sqrt[4]{\frac{H^2}{1 + H^2}} \, \cdot 
  \]
  When $\gamma_1 = \pi / 4$, $H = 0$; we set $F^{\ast} = N$ and obtain
  \[ \mathd N_{01} = - u^{- 2} \mathd F_{01} \text{ and } \mathd N_{02} = v^{-2}
     \mathd F_{02} . \]
  For all values of $H$, this proves that $F^{\ast}$ is a Christoffel dual of $F$,
  and therefore $(F, N)$ has constant mean curvature $H$.
  
  To compute quaternionic cross ratio we assume that $\Phi(\lambda_1)=\1$ and use $\U' \V=\V' \U $:
  \begin{multline*} 
  \text{cr} (F, F_1, F_{12}, F_2) = (F - F_1)  (F_1 - F_{12})^{- 1} (F_{12} - F_2)  (F_2 - F)^{- 1} 
   \\
  \begin{split}
  	& = (\U^{} (\lambda_1)^{- 1} M\U (\lambda_2) - M) \U (\lambda_2)^{- 1}  
    (\V'^{} (\lambda_1)^{- 1} M \V' (\lambda_2) - M)^{- 1} 
    \\
    & = \left( - \frac{2 u \sin (2 \gamma_1)}{\alpha} \i \right) 
    \left( \frac{2 v'}{\beta} \sin (2 \gamma_1) \j \right)^{- 1} 
    \left( - \frac{2 u' \sin (2 \gamma_1)}{\alpha} \i \right) 
    \left( \frac{2 v}{\beta} \sin (2 \gamma_1) \j \right)^{- 1}
    \\
    & =  - \frac{\beta^2}{\alpha^2}  \frac{uu'}{vv'}  \1 = - \frac{\beta^2}{\alpha^2}  \1.
  \end{split}
  \end{multline*}

Formulas (\ref{eq:edge_lengths_H}, \ref{eq:edge_lengths_0}) follow directly from the quaternionic formulas for the corresponding edges derived above.
\end{proof}


\section{From  discrete minimal and CMC surfaces to the Lax pair}

We have seen that the (same) frame $\Phi$ integrating the Lax pair
in~\eqref{eq:Lax-pair}, gives rise to CMC or minimal quad-nets in $\mathbb{R}^3$
and $\mathbb{S}^3$. We shall now prove the converse.

\begin{theorem} \label{thm:reconstruction} 
  For any Q-net of constant mean curvature in $\mathbb{R}^3$ or $\mathbb{S}^3$,
  or minimal in $\mathbb{S}^3$, there exists a Lax pair satisfying
  \eqref{eq:Lax-pair}, such that the immersion formula  \eqref{eq:discrete_immersion_R3} 	
  or \eqref{eq:discrete_immersion_S3} holds.
\end{theorem}

We will prove this theorem in several steps. First we identify geometric quantities $s, A, B$ and algebraic quantities $u, v, \alpha, \beta$ appearing in the corresponding descriptions (\ref{eq:dual_through_s}, \ref{eq:AB_labeling}) and (\ref{eq:Lax-pair}, \ref{eq:alpha2beta2}). We will further denote the edges by $\mathd \hat{F}_{0i}=F_i-F$.

Let us start with the case of discrete CMC surfaces $\hat{F}$ with $H=1$. Its Christoffel dual is given by $\check{F}=\hat{F}^*=\hat{F}+\hat{N}$. The corresponding edges are related by (\ref{eq:dual_through_s}), where $s$ is the conformal metric coefficient of $\hat{F}$. We assume that $s$ alternates its sign along the first direction and does not change the sign along the second direction, i.e. $s_1 s<0,\ s_2 s>0$. Geometrically this means that the quads $(\hat{F},\hat{F}_1,\check{F}_1,\check{F})$ are crossing trapezoids, and the quads $(\hat{F},\hat{F}_2,\check{F}_2,\check{F})$ are embedded trapezoids. 

Comparing the analytic formulas of Theorem~\ref{thm:CMCR3} with the edge length formulas (\ref{eq:AB_labeling}) we obtain the following identification:
\begin{eqnarray} \label{eq:identification}
u^2=-s_1s,\quad v^2=s_2 s \nonumber \\
\alpha(1)=\frac{2}{\sqrt{-A}}, \quad \beta(1)=\frac{2}{\sqrt{B}}. 
\end{eqnarray}

Let us compute the angle between an edge $\mathd \hat{F}_{0i}$ of a discrete CMC surface and its unit normal $\hat{N}$. Since $\mathd \hat{F}_{01}$ and $\mathd
  \hat{F}^{\ast}_{01}$ are parallel (though in opposite directions) and the other
  sides have length 1, one derives 
\begin{equation} \label{eq:<dF,N>_1}
<\mathd \hat{F}_{01},\hat{N}> = \frac{1}{2}\left(\| \mathd \hat{F}_{01} \|+ \| \mathd \hat{F}^{\ast}_{01} \|\right)\| \mathd \hat{F}_{01} \|= \frac{1}{2}A(s_1s-1),
\end{equation}
and similarly,
\begin{equation} \label{eq:<dF,N>_2}
<\mathd \hat{F}_{02},\hat{N}> = \frac{1}{2}\left(\| \mathd \hat{F}_{02} \|- \| \mathd \hat{F}^{\ast}_{02} \|\right)\| \mathd \hat{F}_{02} \|= \frac{1}{2}B(s_2s-1).
\end{equation}
Here we have used (\ref{eq:dual_through_s},\ref{eq:AB_labeling}).

We will use the following technical lemma, which can be easily checked.
 \begin{lemma} \label{lemma:commutation-lambda}
    The commutation property (\ref{eq:discrete_Lax}) for all $\lambda$ is 
    equivalent to 
    the condition $u u' = v v'$ (equation \eqref{eq:uuprimevvprime}), 
    together with the commutation property for matrices and its derivative 
    evaluated \emph{only} at $\lambda = 1$:
    \[
    \V'(1) \U(1) = \U'(1) \V(1) \; \text{ and } \;
    \dot\V'(1) \U(1) + \V'(1) \dot\U(1) = \dot\U'(1) \V(1) + \U'(1) \dot\V(1)
    \] 
    It is also equivalent to $u u' = v v'$ together with the commutation
    property evaluated at two values $\lambda,\lambda^{-1}$, provided 
    $\lambda \notin {\pm 1, \pm i}$.
  \end{lemma}

The following lemmas prove the existence and uniqueness of the Lax pair for a given quad of a discrete CMC surface.

\begin{lemma}[$\mathbb{R}^3$ version]  \label{lemma:quad-reconstruction-R3} 
  Let $Q = (\hat{F}, \hat{F}_1, \hat{F}_{12}, \hat{F}_2)$ be a CMC-1 quad in
  $\mathbb{R}^3$ with Gauss map $(\hat{N}, \hat{N}_1, \hat{N}_{12},
  \hat{N}_2)$. Let $\Phi$ be any frame for $\hat{N}$, namely $\hat{N} = -
  \Phi^{- 1} \k \Phi$ (such a frame at the point $\hat{F}$ is determined up to $U (1)$
  action). 
  Then there exist $\U(\lambda), \V(\lambda), \U'(\lambda),
  \V'(\lambda)$ satisfying (\ref{eq:discrete_Lax}), and thus generating the quad, together with its Gauss map by formulas (\ref{eq:discrete_immersion_R3}).
\end{lemma}

\begin{proof}
  In the following, $\mathcal{U}$, $\alpha$, etc. will denote the usual quantities evaluated 
  at $\lambda=1$, and we will write specifically $\mathcal{U}(\lambda)$ when considering the loop.
  
  As we have demonstrated above (geometric) CMC-1 quad $Q$ in $\mathbb{R}^3$, together with its Gauss map determines 
  $\alpha, \beta, u, v, u', v'$ so only $a, b$ remain to be found\footnote{Although $| a |$ and $| b |$ are given by
  \eqref{eq:alpha2beta2}.}. The Lax matrix $\U$ can be determined uniquely 
  from $\Phi$, through equation~\eqref{eq:def-U}:
  \[ 
  	  \Phi \mathd \hat{F}_{01} \Phi^{- 1} = - \frac{2 u}{\alpha} \i \U
	  = \frac{2i u}{\alpha^2} \begin{pmatrix}
       u + u^{- 1} &  \bar{a}\\
       a & - (u + u^{- 1})
     \end{pmatrix}. 
  \] 
  Indeed, the length of the edge is $| 2 u / \alpha |$ by (\ref{eq:identification}), so that $a$ is here 
  only to fix the directions in the $(\i, \j)$ plane, provided the coordinate along 
  the $(-\k)$ axis (i.e. the projection of $\mathd \hat{F}_{01}$ along $\hat{N}$) is
  \[ 
  \frac{2 u}{\alpha}  \frac{u + u^{- 1}}{\alpha}. 
  \]
  This fact follows from  (\ref{eq:<dF,N>_1}). 
  So that $a$ is now fixed, though it depends on the $\Phi$ gauge. Similarly, 
  $\mathd \hat{F}_{02}$ determines $\V$  and $b$. 


  Let us now check that $\Phi_1$ and $\Phi_2$ defined through $\U$ and $\V$ 
  are frames for $\hat{N}_1$ and $\hat{N_2}$ respectively. By property of the Koenigs net,
  $\hat{N}_1 - \hat{N} =  - (1+u^{-2}) \mathd \hat{F}_{01}$, and, conjugating by $\Phi$,
  we need to check that
  \[
  - \U^{-1} \k \U + \k = - \frac{2}{\alpha} \frac{u^2+1}{u} \i \U
  \]
  which holds due to the specific form of $\U$, and similarly for $\V$.
  From $\Phi_1$ and $\Phi_2$, we derive $\U',\V'$ as above, with their specific form. 
  
  Having defined all their coefficients, the Lax matrices are fully determined, 
  and there remains only to check the commutation property (\ref{eq:discrete_Lax}) for all $\lambda$. 
  This can be done using Lemma~\ref{lemma:commutation-lambda}.

  By (\ref{eq:identification}) $u u' = v v'$ holds. Comparing the cross ratio written in terms of 
  $\U$, $\V$, $\U'$, $\V'$ and its known value $-\beta^2/\alpha^2$ proves 
  the commutation for $\lambda=1$ (reverse the proof of Theorem~\ref{thm:CMCR3}). 
  This shows also that $\hat{N}_{12} = - \Phi_{12}^{-1} \k \Phi_{12}$, 
  where $\Phi_{12} = \V' \U \Phi = \U' \V \Phi$.
  The derived commutation property is a consequence of the additive commutation relation 
  \[
  \mathd \hat{F}_{01} + \mathd \hat{F}_{1, 12} = \mathd \hat{F}_{02} + \mathd \hat{F}_{2, 12} \, .
  \]
  For simplicity, we apply it to $G = \hat{F} + \frac{1}{2}  \hat{N} 
  = - \Phi^{- 1}  \frac{\partial \Phi}{\partial \gamma}_{| \gamma = 0 }$, which is a well-defined 
  quad, and hence closes:
  \begin{eqnarray*}
  0 & = & \Phi (\mathd G_{01} + \mathd G_{1, 12} - \mathd G_{02} - \mathd
  G_{2, 12}) \Phi^{- 1}\\
  & = & \mathcal{V}^{- 1}  \dot{\mathcal{V}} +\mathcal{V}^{- 1}
  \mathcal{U}'^{- 1}  \dot{\mathcal{U}}' \mathcal{V}-\mathcal{U}^{- 1} 
  \dot{\mathcal{U}} -\mathcal{U}^{- 1} \mathcal{V}'^{- 1}  \dot{\mathcal{V}}'
  \mathcal{U}\\
  & = & \mathcal{U}^{- 1} \mathcal{V}'^{- 1}  (\mathcal{U}' 
  \dot{\mathcal{V}} + \dot{\mathcal{U}}' \mathcal{V}-\mathcal{V}'
  \dot{\mathcal{U}} - \dot{\mathcal{V}}' \mathcal{U})
  \end{eqnarray*}
  where we use twice the commutation at order zero.
  
  We have thus determined the Lax matrices (uniquely, once a frame $\Phi$ at $\hat{F}$ is set).
  The CMC-1 quad they generate is the one we started from.
  Note that, although we have started as usual with the lower left vertex, this
choice plays no role, and we might have fixed $\Phi_1$, $\Phi_2$ or $\Phi_{12}$ 
and recovered the rest similarly. 
\end{proof}

\begin{lemma}[$\mathbb{S}^3$ version]
  \label{lemma:quad-reconstruction-S3} 
  Let $Q = (F, F_1, F_{12}, F_2)$ be a CMC or minimal quad in $\mathbb{S}^3$ 
  with Gauss map $(N, N_1, N_{12}, N_2)$. Let $(\phi, \phi')$ be any frame 
  at the vertex $F$, i.e. any couple in $SU (2)$ such that $F = \phi'^{- 1} M \phi$ 
  and $N = - \phi'^{- 1} \k M \phi$ (such a pair is determined up to $U(1)$ 
  action). Then there exist $\U(\lambda), \V(\lambda), \U'(\lambda),
  \V'(\lambda)$ satisfying (\ref{eq:discrete_Lax}), and thus generating the quad, together with its Gauss map by formulas (\ref{eq:discrete_immersion_S3}).
\end{lemma}

\begin{proof}
	The proof goes along the same lines as in Lemma \ref{lemma:quad-reconstruction-R3}, 
	with a few specificities; in particular, whereas in $\mathbb{R}^3$ scaling allows 
	us to freely set the mean curvature to $1$, in $\mathbb{S}^3$, we use 
	the mean curvature to determine $\lambda_1 = e^{i \gamma_1}, \lambda_2 = e^{i \gamma_2}$ 
	as in theorem \ref{thm:CMCS3}: $\gamma_1 = - \gamma_2 = \frac{1}{2} \mathrm{arccot} H$
  ($\gamma_1 = \pi / 4$ if $H = 0$), $\gamma_1$ being taken in $[0,\pi/2]$.
	A (geometric) CMC-$H$ quad $Q$ in $\mathbb{S}^3$,
  together with its Gauss map, comes with a discrete conformal metric $s$ and canonical
  edge labelings $A, B$, such that the edge lengths in both directions are
  $Ass_1$ and $Bss_2$ respectively. This allows us to determine $u, v$ by
  \begin{equation}  \label{eq:u_v_s}
   u^2 = - ss_1  \sqrt{\frac{1 + H^2}{H^2}} = - \frac{ss_1}{\cos (2
     \gamma_1)}, \quad v^2 = ss_2  \sqrt{\frac{1 + H^2}{H^2}} =
     \frac{ss_2}{\cos (2 \gamma_1)} \text{ if } H \neq 0, 
 \end{equation}
  \[ u^2 = - ss_1, \quad v^2 = ss_2 \quad \text{ if } H = 0, 
  \]
  and similarly $u',v'$. We set positive $\alpha, \beta$ such that 
  \begin{equation}	\label{eq:alpha_A}
  \alpha^2 = - \frac{4 \sin^2 
  (2 \gamma_1)}{A \cos (2 \gamma_1)},\quad \beta^2 = \frac{4 \sin^2 (2
  \gamma_1)}{B \cos (2 \gamma_1)}
  \text{ if } H \neq 0,
  \end{equation} 
  $$
 \alpha^2 = - \frac{4}{A},\quad \beta^2
  = \frac{4}{B} \text{ if } H = 0.
  $$ 
  In the following, we will favor the
  notations in terms 
  of $\gamma_1$:
  \begin{align*}
  \phi' \mathd F_{01} \phi^{- 1} 
  &= \U (\lambda_1)^{- 1} M \U (\lambda_2) - M 
  = \U (\lambda_1)^{- 1}  \left( - \frac{2 u \sin (2 \gamma_1)}{\alpha} \i \right) 
  \\
  & = \frac{2 u \sin (2 \gamma_1)}{\alpha}  \frac{1}{\alpha}  \begin{pmatrix}
       i (\lambda_1 u + \lambda_1^{- 1} u^{- 1}) & \bar{a}\\
       a & - i (\lambda_1 u^{- 1} + \lambda_1^{- 1} u)
     \end{pmatrix}
  \end{align*}
  We recognize the length of the edge, which proves incidentally that the right
  hand side matrix is unimodular, and so is $\U (\lambda_1)$. As in
  $\mathbb{R}^3$, $a$ is given by the $(\i,
  \j)$ component, provided the $(\1,
  \k)$ component is correct. The latter is equivalent to the following two conditions:
  \begin{enumerate}
    \item the angle $\theta$ between $\mathd F_{01}$ and $N$ satisfies
    \begin{align} \label{eq:cos_theta}
      \cos \theta 
      & =  \frac{\langle \mathd F_{01}, N \rangle}{\| \mathd F_{01} \|} 
      = \frac{1}{\alpha}  (u^{- 1} + \cos (2 \gamma_1) u) \, ,
    \end{align}
    \item the angle $\chi$ between $\mathd F_{01}$ and $F$ satisfies
    \begin{align} \label{eq:cos_xi}
      \cos \chi 
      &= \frac{\langle \mathd F_{01}, F\rangle }{\| \mathd F_{01} \|}  
      = - \frac{u \sin (2 \gamma_1)}{\alpha} \, \cdotp
    \end{align}
  \end{enumerate}
  
  Geometric derivation of $\cos\theta$ is analogous to (\ref{eq:<dF,N>_1}). 
  In particular in the CMC case $H=\cot(2\gamma_1)\neq 0$ the dual isothermic surface 
  of $F$ is $F+\frac{1}{H}N$, and (\ref{eq:<dF,N>_1}) is modified to
  $$
  \frac{2}{H}\cos \theta=\|\mathd F_{01}\|+\|\mathd F^*_{01}\|=\|\mathd F_{01}\|(1+\frac{1}{s_1s})=\sqrt{s_1sA}(1+\frac{1}{s_1s}). 
  $$
 Substituting (\ref{eq:u_v_s}),(\ref{eq:alpha_A}) we arrive at (\ref{eq:cos_theta}).
 
%
  
  Identity (\ref{eq:cos_xi}) follows directly from $\| F \| = \| F_1 \| = 1$:
  \[ \cos \chi = \langle F, \mathd F_{01} \rangle = - \frac{1}{2}  \| \mathd
     F_{01} \|^2 = - \frac{u \sin (2 \gamma_1)}{\alpha}. \]
  Along the other coordinate, we have $v$ and non-crossing trapezoids, but the
  reasoning is analogous, and fixes $b$.

	Setting $\Phi(\lambda_1) = \phi'$ and $\Phi(\lambda_2) = \phi$, 
	 we check that 
	$N_1 = - \Phi_1(\lambda_1)^{- 1} \k M \Phi_1(\lambda_2)$, where
	$\Phi_1(\lambda_i) = \U(\lambda_i) \Phi(\lambda_i)$. Indeed, this amounts to reversing 
	the proof in Theorem \ref{thm:CMCS3}. We have the analogous result for $N_2$, 
	which allows us to compute $a',b'$ and therefore $\U'(\lambda)$ and $\V'(\lambda)$.

	To prove (\ref{eq:discrete_Lax}) we reverse the calculation in the proof of
	Theorem~\ref{thm:CMCS3}, which shows that $\V'(\lambda_1) \U(\lambda_1) 
	= \U'(\lambda_1) \V(\lambda_1)$. 
	
	To prove the analogous result at $\lambda_2$, 
	we consider the symmetric (see (\ref{eq:quaternionic_S^3})) surface $F^{-1}$ with the normal $N^{-1}$. This exchanges $\lambda_1$ with $\lambda_2$ while
	preserving all the metric and affine properties. We conclude with 
	Lemma~\ref{lemma:commutation-lambda}.
	
\end{proof}

\begin{proof}[Proof of the Theorem]~\\
  The same strategy works for the Euclidean and spherical nets, so we will 
  describe it in the Euclidean case. We start by constructing the Lax pair 
  on one quad $Q = (\hat{F}, \hat{F}_1, \hat{F}_{12}, \hat{F}_2)$. 
  For any choice of compatible
  frame $\Phi$, i.e. $\hat{N} = - \Phi^{- 1} \k \Phi$ at the base
  vertex, we can find a (unique) Lax pair generating $Q$, according to the
  Lemmas \ref{lemma:quad-reconstruction-R3} or
  \ref{lemma:quad-reconstruction-S3}. This in turn determines a Lax pair on
  adjacent faces, which generates the corresponding quads, and therefore the
  whole Q-net.
  
  This reasoning holds provided that we do not obtain a contradiction when
  the Gauss map is given at more than one vertex. E.g. once $Q = (\hat{F},
  \hat{F}_1, \hat{F}_{12}, \hat{F}_2)$ has been constructed, we may construct
  the adjacent quad $Q_1 = (\hat{F}_1, \hat{F}_{11}, \hat{F}_{112},
  \hat{F}_{12})$ by starting with $\hat{N}_1$ or with $\hat{N}_{12}$. However
  the two quads constructed this way must agree, because $\hat{N}_1$ fully
  determines $\hat{N}_{12}$ (knowing the corresponding edge), see lemma
  \ref{lemma:quad-reconstruction-R3}.
\end{proof}

\medskip

\section{The discrete Lawson correspondence}

The results shown above allow us to define a discrete Lawson correspondence between 
Q-nets in $\mathbb{S}^3$ and $\mathbb{R}^3$.
\begin{theorem}   \label{thm:discrete_Lawson}
Let $F$ be a minimal Q-net in $\mathbb{S}^3$ with discrete conformal metric $s$. 
Then there exists a constant mean curvature Q-net $\hat{F}$ in $\mathbb{R}^3$ 
with constant mean curvature $H = 1$ and the same discrete conformal metric $s$. 
More generally, this correspondence takes any Q-net of constant mean curvature $H$ 
lying in the sphere of curvature $\kappa$ to a Q-net of constant mean curvature $H'$ 
in the sphere of curvature $\kappa'$, provided $H^2 + \kappa = H'^2 + \kappa'$. 
The Euclidean case corresponds to $\kappa'=0$. Additionally, the two Q-nets 
$F$ and $\hat{F}$ are given by formulas (\ref{eq:discrete_immersion_S3}) and 
(\ref{eq:discrete_immersion_R3}) with the same Lax matrices $\U(\lambda)$ and $\V(\lambda)$.
\end{theorem}
\begin{proof} 
The construction is a direct consequence of the previous theorems, 
but to see it clearly, we shall state a slightly modified version of
Theorems~\ref{thm:CMCS3} and \ref{thm:reconstruction}.

Let us start with the claim about Q-nets in spheres with different curvatures.
Any Lax pair as defined in (\ref{eq:Lax-pair}), and any choice
$\lambda_1 = e^{i \gamma_1}$ of the spectral parameter gives rises
to an constant mean curvature Q-net $(F^{\gamma_1},N)$ in the sphere
of curvature $\kappa = \sin^2(2 \gamma_1)$, obtained as a a scaled up version 
of the spherical net $F$ defined in Theorem~\ref{thm:CMCS3}:
\[
F^{\gamma_1} = \frac{1}{\sin (2 \gamma_1)} F 
= \frac{1}{\sin (2 \gamma_1)} \Phi (\lambda_1)^{- 1} M \Phi (\lambda_1^{- 1}),
\]
and $N$ remains the same. The mean curvature is $H = \cos (2 \gamma_1)$.
If $\gamma_1 = \pi/4$, this is the minimal Q-net in $\mathbb{S}^3$.
\\

By the same reasoning according to Theorems~\ref{thm:CMCS3} and \ref{thm:reconstruction} there exists infinitely many other
CMC Q-nets $(F^{\gamma'_1},N)$ of constant mean curvature
$H' = \cos (2 \gamma'_1)$ in the sphere of curvature $\kappa'=\sin^2(2 \gamma'_1)$ with the same coefficients $u,v$ in the Lax pair.

One checks easily that 
\[
H'^2 + \kappa' = \cos^2 (2 \gamma'_1) + \sin^2(2 \gamma'_1) = 1
= H^2 + \kappa.
\]
The most general case, where $H^2+\kappa \neq 1$ is obtained from the latter by 
direct scaling.
Finally, since the discrete conformal metric coefficient is defined geometrically up to general factor (see Sect.~\ref{sec:discrete_CMC_Q-nets}), the coincidence of the Lax matrix coefficients $u,v$ is equivalent to the preservation of the discrete conformal metric coefficient $s$.

The case of $\kappa'=0$ is dealt in exactly the same way, except that 
the immersion formula in Theorem~\ref{thm:CMCR3} is used for $\mathbb{R}^3$. 
We note now that the factorization formulas \eqref{eq:vertexmetricR3} and \eqref{eq:vertex-metric-S3} for minimal nets in $\mathbb{S}^3$ and CMC-$1$ nets in $\mathbb{R}^3$ coincide. The coincidence of the coefficients $u,v$ implies that the discrete conformal metric factors $s$ determined from these formulas are identical, in the same way as, in the smooth case, both surfaces are isometric.  

\end{proof}
\begin{remark}
Furthermore, the Q-net in $\mathbb{R}^3$ is the limit of the spherical CMC
nets, when the radius of the sphere increases to infinity, keeping the point
$\1$ fixed. Indeed, let $\gamma_1$ tend to zero. Then $H = \cos (2 \gamma_1)$
goes to $1$, and
\begin{align*}
  F - \1 & = \Phi (\lambda_1)^{- 1} M \Phi (\lambda_1^{- 1}) 
  = (\Phi + \gamma_1  \dot{\Phi} + o (\gamma_1))^{- 1}  (\1 +
  \gamma_1  \k + o (\gamma_1))  (\Phi - \gamma_1  \dot{\Phi} + o
  (\gamma_1)) - \1
  \\
  & = \gamma_1  (- 2 \Phi^{- 1}  \dot{\Phi} + \Phi^{- 1}  \k \Phi) + o(\gamma_1)
\end{align*}
so
\[ 
\frac{1}{\sin (2 \gamma_1)}  (F - \1) \sim \frac{\gamma_1}{2 \gamma_1}  
(- 2 \Phi^{- 1}  \dot{\Phi} + \Phi^{- 1}  \k \Phi) 
= - \Phi^{- 1}  \dot{\Phi} + \frac{1}{2} \Phi^{- 1}  \k \Phi = \hat{F} \, .
\]
\end{remark}

\bigskip

\begin{remark} \label{Calapso}
\em
At last, let us remark that this definition of the Lawson correspondence 
coincides with the seemingly very different one proposed in \cite{BuHRS,BuHR}, 
based on the Calapso transform and the conserved quantities formalism. 
We will not go into the details of the latter, but we will show that 
both definitions agree. Indeed, and despite their very different formulations, 
it suffices to show that the mean curvature, the discrete conformal metric and the labelings
change in the same way.

\medskip

Let $H=\cos(2 \gamma_1)$ and $H' = \cos(2 \gamma'_1)$ be the two corresponding mean curvatures 
as defined above. Since $u,v$ are common to both immersions, and $s s_1 = - u^2 H$ 
(resp. $s s_2 = v^2 H$), then $r = s H^{-1/2}$ takes the same values for both surfaces.
The Reader may check that this vertex function $r$ is the one used in \cite[\S 3 and \S 4.2]{BuHR},
which is invariant under the Calapso transform.
The edge labelings are $-\alpha(\gamma_1)^{-2},\beta(\gamma_1)^{-2}$ and
$-\alpha(\gamma_1)^{-2},\beta(\gamma_1)^{-2}$, up to a multiplicative constant~$c$. 
From~\eqref{eq:alpha2beta2}, we see that 
\[
\alpha(\gamma'_1)^2 - \alpha(\gamma_1)^2 = \cos(2 \gamma'_1) - \cos(2 \gamma_1) = H' - H,
\]
and similarly $\beta(\gamma'_1)^2 - \beta(\gamma_1)^2 = H - H'$. Choosing $c=-1$,
the edge labeling satisfy
\[
a_{01}' = \frac{1}{\alpha'^2} = \frac{1}{\alpha^2 + H' - H} =
\frac{1}{\frac{1}{a_{01}} + H' - H} = \frac{a_{01}}{1 + (H' - H) a_{01}},
\]
and similarly $a_{02}' = \frac{a_{02}}{1 - a_{02}  (H - H')}$, which is again the prescribed
behavior of the Calapso transform. Therefore both correspondences agree.
\end{remark}

\medskip

\paragraph{Conclusion and open questions}
~\\

In this paper we have established a discrete Lawson isometry between discrete isothermic minimal surfaces in ${\mathbb S}^3$ and discrete isothermic CMC surfaces in ${\mathbb R}^3$. The isometry is understood in the sense that both corresponding isothermic surfaces have the same discrete conformal metric coefficient. It is appealing to lift this correspondence to the level of frames as in the smooth case (Theorem~\ref{thm:Lawson_smooth}). However the isothermic parametrizations in \eqref{eq:Lawson_smooth} do not correspond: an isothermic surface in ${\mathbb S}^3$ corresponds to a CMC surface from the associated family. 

One way to reach that, and an important achievement by itself, would be to introduce geometrically a discrete metric for the associated families in ${\mathbb R}^3$ and ${\mathbb S}^3$ that generalizes the conformal metric coefficient $s$ of isothermic surfaces. It should be the same for the whole associated family. On the level of the Lax representation it is the coefficient $w$ in this paper. The next step then would be to find a discrete analogue of \eqref{eq:Lawson_smooth}.

Some progress in this direction has been achieved in \cite{HSFW} where the associated family in ${\mathbb R}^3$ was described as edge-constraint net with non-planar faces (see Section~\ref{sec:discrete_CMC_R3}). The curvatures were defined there but not the conformal metric. Here we are dealing with discrete isothermic surfaces in non-isothermic parametrization. Let us mention that such triangulated surfaces were recently introduced in \cite{LP}.


\end{document}